\documentclass[11pt]{article}

\usepackage{amsfonts,amsmath,latexsym,color,epsfig}
\newtheorem{theorem}{Theorem}[section]
\newtheorem{lemma}{Lemma}[section]

\newtheorem{conjecture}{Conjecture}

\newcommand{\definition}
      {\medskip\noindent {\bf Definition\hspace{1mm}}}
\newenvironment{proof}
      {\medskip\noindent{\bf Proof:}\hspace{1mm}}
      {\hfill$\Box$\medskip}

\input{epsf}

\makeatletter
\def\Ddots{\mathinner{\mkern1mu\raise\p@
\vbox{\kern7\p@\hbox{.}}\mkern2mu
\raise4\p@\hbox{.}\mkern2mu\raise7\p@\hbox{.}\mkern1mu}}
\makeatother

\def\d{\delta}
\def\e{\epsilon}

\def\mG{\mathcal{G}}
\def\mH{\mathcal{H}}

\title{\vspace{-0.7cm}Erd\H{o}s-Hajnal-type theorems in hypergraphs}
\author{David Conlon\thanks{St John's College, Cambridge CB2 1TP, United
Kingdom.
Email: {\tt
d.conlon@dpmms.cam.ac.uk}. Research supported by a Royal Society University
Research Fellowship.} \and Jacob Fox\thanks{Department of
Mathematics, MIT, Cambridge, MA 02139-4307. Email: {\tt
fox@math.mit.edu}. Research supported by a Simons Fellowship.} \and
Benny Sudakov\thanks{Department of Mathematics,
UCLA,  Los Angeles, CA 90095. Email: {\tt bsudakov@math.ucla.edu}. Research
supported in part by NSF CAREER award DMS-0812005 and by a
USA-Israeli BSF grant.}}
\date{}
\begin{document}
\maketitle

\begin{abstract}
The Erd\H{o}s-Hajnal conjecture states that if a graph on $n$ vertices is
$H$-free, that is, it does not contain an induced copy of a given graph $H$,
then it must contain either a clique or an independent set of size $n^{\d(H)}$,
where $\d(H) > 0$ depends only on the graph $H$. Except for a few special
cases, this conjecture remains wide open. However, it is known that a $H$-free
graph must contain a complete or empty bipartite graph with parts of polynomial
size.

We prove an analogue of this result for $3$-uniform hypergraphs, showing that
if a $3$-uniform hypergraph on $n$ vertices is $\mH$-free, for any given $\mH$,
then it must contain a complete or empty tripartite subgraph with parts of
order $c (\log n)^{\frac{1}{2} + \d(\mH)}$, where $\d(\mH) > 0$ depends only on
$\mH$. This improves on the bound of $c (\log n)^{\frac{1}{2}}$, which holds in
all $3$-uniform hypergraphs, and, up to the value of the constant $\d(\mH)$, is
best possible.

We also prove that, for $k \geq 4$, no analogue of the standard
Erd\H{o}s-Hajnal conjecture can hold in $k$-uniform hypergraphs. That is, there
are $k$-uniform hypergraphs $\mH$ and sequences of $\mH$-free hypergraphs which
do not contain cliques or independent sets of size appreciably larger than one
would normally expect.
\end{abstract}

\section{Introduction}

One version of Ramsey's theorem states that, for any $s$ and $t$, there is a
natural number $n$ such that, in any $2$-colouring of the edges of the complete
graph $K_n$, there
is either a red $K_s$ or a blue $K_t$. The smallest number $n$ for which this
holds is the Ramsey number $r(s,t)$. A classical result of Erd\H{o}s and
Szekeres \cite{ES35} implies that, for $s \geq t$, we have
$$r(s, t) \leq \binom{s+t-2}{t-1} \leq s^{t-1}.$$
Phrasing this differently, we see that every graph on $n$ vertices which
contains no copy of $K_t$ must have an independent set of size $n^{1/(t-1)}$.

We call a graph $H$-free if it contains no induced copy of a given graph $H$.
The famous conjecture of Erd\H{o}s and Hajnal \cite{EH89} states that if a
graph on $n$ vertices is $H$-free then it must contain either a clique or an
independent set of size $n^{\d(H)}$, where $\d(H) > 0$ depends only on the
graph $H$. The observations of the last paragraph imply that this conjecture is
true when $H$ is a complete graph. Suppose now that we know the conjecture to
be true for two graphs $H_1$ and $H_2$. Erd\H{o}s and Hajnal showed that the
conjecture will then also hold for a graph $H$ formed by connecting the
vertices of $H_1$ and $H_2$ by either a complete or an empty bipartite graph.
This easily allows one to prove the conjecture for all graphs with at most $4$
vertices except $P_4$, the path with $3$ edges.

A more general result, due to Alon, Pach and Solymosi \cite{APS01}, says that
if $H_1$ and $H_2$ have the Erd\H{o}s-Hajnal property then so does the graph
$H$ formed by blowing up a vertex of $H_1$ and replacing it with a copy of
$H_2$. The conjecture is also known for some graphs which do not fall naturally
into these inductively defined classes (but which may be added as base cases to
widen the range of these classes). For example, Erd\H{o}s and Hajnal noted that
since graphs with no induced $P_4$ are perfect the conjecture is true for
$P_4$. More recently, Chudnovsky and Safra \cite{CS08} have shown that the
conjecture holds for the bull, a self-complementary graph on $5$ vertices
consisting of a triangle with two pendant edges. The smallest cases that now
remain open are the cycle $C_5$ and the path $P_5$.

For general $H$, Erd\H{o}s and Hajnal proved that if a graph on $n$ vertices is
$H$-free then it must contain a clique or an independent set of size $e^{c(H)
\sqrt{\log n}}$. This is a significant improvement over the bound of $c \log n$
which, by Ramsey's theorem, holds in all graphs, but it is still quite far from
the conjecture. However, as observed in \cite{EHP00}, their method does allow
one to find complete or empty bipartite subgraphs each side of which are of
polynomial size. Recently, Fox and Sudakov \cite{FS09} went one step further by
proving that there is either a complete bipartite graph or an independent set
of polynomial size. For further related results, see \cite{APPRS05, CZ11,
FPT10, FS08, LRSTT10, PR99}.

For $3$-uniform hypergraphs, Erd\H{o}s and Rado \cite{ER52} proved that, in any
$2$-colouring of the edges of the complete graph $K_n^{(3)}$ on $n$ vertices,
there is a monochromatic clique of size $c \log \log n$. Phrased differently,
this says that any $3$-uniform hypergraph on $n$ vertices contains either a
clique or an independent set of size $c \log \log n$. Given the situation for
graphs, it is tempting to conjecture that if a $3$-uniform hypergraph on $n$
vertices is $\mH$-free, for some given $\mH$, then there should be a clique or
an independent set of size much larger than $\log \log n$. We feel, but have
been unable to prove, that for general $\mH$ this may be too much to expect.

However, at least for some excluded hypergraphs, an Erd\H{o}s-Hajnal-type
estimate does appear to hold. Let $r_3(s,t)$ be the smallest number $n$ such
that, in any $2$-colouring of the edges of $K_n^{(3)}$, there is a red
$K_s^{(3)}$ or a blue $K_t^{(3)}$. A recent result of the authors \cite{CFS10},
improving on the work of Erd\H{o}s and Rado, implies that, for $t$ fixed,
\[r_3(s,t) \leq 2^{c_t s^{t-2} \log s}.\]
It follows that any $3$-uniform hypergraph with $n$ vertices which does not
contain a copy of $K_t^{(3)}$ must contain an independent set of size roughly
$\left( \log n / \log \log n\right)^{1/(t-2)}$.

Even this is only progress in a weak sense. It is not yet known whether $\log
\log n$ is the correct order of magnitude for Ramsey's theorem in $3$-uniform
hypergraphs. Indeed, this question has become rather notorious, Erd\H{o}s
offering \$500 dollars for its resolution. The best construction, given by a
random hypergraph, only tells us that there are hypergraphs on $n$ vertices
with no
clique or independent set of size $c (\log n)^{\frac{1}{2}}$. So we cannot say
with certainty whether the estimate for $K_t^{(3)}$-free graphs really does
improve on the general bound.

Given this state of affairs, we follow a different route, suggested by R\"odl and Schacht,
attempting to extend the bipartite counterpart of the Erd\H{o}s-Hajnal theorem
to tripartite $3$-uniform hypergraphs. In any given $3$-uniform hypergraph on
$n$ vertices, one may always find a complete or empty tripartite subgraph with
parts of order at least $c (\log n)^{\frac{1}{2}}$. This follows from a
standard extremal result due to Erd\H{o}s \cite{E64}. We improve this result
for $\mH$-free graphs.

\begin{theorem} \label{TripIntro}
Let $\mH$ be a $3$-uniform hypergraph. Then there exists a constant $\d(\mH) >
0$ such that, for $n$ sufficiently large, any $\mH$-free $3$-uniform hypergraph
on $n$ vertices contains a complete or empty tripartite subgraph each part of
which has order at least $(\log n)^{\frac{1}{2} + \d(\mH)}$.
\end{theorem}

This improves upon a result of R\"odl and Schacht \cite{R10}, who used the
regularity method for hypergraphs to show that the size of the largest complete
or empty tripartite subgraph grows arbitrarily faster than the function $(\log
n)^{\frac{1}{2}}$. However, because their results depend upon the regularity
lemma, their result does not provide good bounds on the multiplying function.

Our result, on the other hand, is not far from best possible, since, for many
$\mH$, one cannot do better than $c \log n$. To see this, consider the random
graph on the vertex set $\{1,2, \dots, n\}$ where each edge is chosen with
probability $\frac{1}{2}$. For $c$ sufficiently large, with high probability,
this graph contains no complete or empty bipartite graph with parts of order $c
\log n$. Fix such a graph and call it $G_n$. Let $\mG_n$ be the $3$-uniform
hypergraph on the same vertex set whose edge set consists of all those triples
$(i_1, i_2, i_3)$ with $i_1 < i_2 < i_3$ such that $(i_1, i_2)$ is an edge in
$G_n$.

It is easy to see that $\mG_n$ contains no complete or empty tripartite
subgraph
with sets of size $c \log n$. Suppose otherwise and let $U, V$ and $W$ be
subsets of size $c \log n$ which define such a tripartite graph. Without loss
of generality, assume that the largest vertex $w$, in the ordering inherited
from the integers, lies in $W$. Then, by construction, there must be a complete
or empty bipartite subgraph in $G_n$ between $U$ and $V$, a contradiction. Now,
for any subset $X$ of the vertices of $\mG_n$, let $x_1$ and $x_2$ be the two
smallest vertices in $X$. Then, again by construction, for every $x \in
X\char92\{x_1, x_2\}$, either all edges of the form $(x_1, x_2, x)$ are in
$\mG_n$ or none of them are. Choose a small hypergraph $\mH$ containing no
vertex
pair $(x_1, x_2)$ with this property. For example, one may take $\mH$ to be a
tight cycle on five vertices, that is, with vertices $\{1, 2, 3, 4, 5\}$ and
edge set $\{123, 234, 345, 451, 512\}$. Then $\mG_n$ is $\mH$-free. Since it
also
contains no tripartite subgraph of size $c \log n$, this completes our claim.

The proof of Theorem \ref{TripIntro} relies upon a new embedding lemma which
says that if the edges of $\mG$ are fairly well-distributed, in the sense that
in any graph containing many triangles a positive proportion of these triangles
form edges both of $\mG$ and of its complement $\overline{\mG}$, then one may embed an
induced copy of any particular small hypergraph $\mH$. If the hypergraph is not
well-distributed in the sense described above, then it turns out that it must
contain a complete or empty tripartite subgraph which is much larger than one
would normally expect.

Despite this description being a reasonable one for any uniformity, the proof
does not extend to the $k$-uniform case for any $k \geq 4$. We will say more
about this in the concluding remarks. For now, we will return to considering
the analogue of the usual Erd\H{o}s-Hajnal problem in hypergraphs. That is,
given a $\mH$-free $k$-uniform hypergraph, how large of a clique or independent
set must it contain?

Let $r_k(\ell)$ be the diagonal Ramsey function, that is, the minimum $n$ such
that in any $2$-colouring of the edges of $K_n^{(k)}$ there is a monochromatic
copy of $K_\ell^{(k)}$. The tower function $t_k(x)$ is defined by $t_1(x)=x$
and $t_{i+1}(x)=2^{t_i(x)}$. A result of Erd\H{o}s and Rado implies that, for
$k \geq 4$,
\[r_k(\ell) \leq t_{k-2} ((r_3(c_k \ell))^3).\]
On the other hand, a result of Erd\H{o}s and Hajnal (see \cite{CFS1103,
GRS90}), referred to as the stepping-up lemma, allows one, for $k \geq 4$, to
take a colouring of the $(k-1)$-uniform hypergraph on $n$ vertices containing
no monochromatic cliques of size $\ell$ and to show that there is a colouring
of the $k$-uniform hypergraph on $2^n$ vertices containing no monochromatic
clique of size $2\ell + k - 5$. In particular, this can be used to show that,
for
$\ell$ sufficiently large,
\[r_k(\ell) \geq t_{k-2} (r_3(c'_k \ell)).\]
Therefore, once the asymptotic behaviour of $r_3(\ell)$ is understood, so is
that of $r_k(\ell)$. To be more precise, let $r_k^{-1}$ be the inverse function
of $r_k$. Restating the results quoted above and using the fact that
$r_3(\ell)$ is at least exponential in $\ell$ tells us that, for $k \geq 3$ and
$n$ sufficiently large depending on $k$, $r_k^{-1} (n)$ has upper and lower
bounds of the form
\[a'_k r_3^{-1} (\log_{(k-3)} n) \leq r^{-1}_{k} (n) \leq a_k r_3^{-1}
(\log_{(k-3)} n),\]
where $\log_{(0)} (x) = x$ and, for $i \geq 1$, $\log_{(i)}(x) = \log
(\log_{(i-1)} (x))$ is the iterated logarithm.

For an Erd\H{o}s-Hajnal-type theorem to hold, we would therefore need that
whenever a hypergraph on $n$ vertices is $\mH$-free there is a clique or
independent set of size much larger than $r_3^{-1} (\log_{(k-3)} n)$. We will
disprove this by showing that there are already simple examples of hypergraphs
$\mH$ which are not contained in step-up colourings. This implies the following
theorem.

\begin{theorem} \label{StepUpIntro}
For $k \geq 4$, there exists a constant $c_k$, a $k$-uniform hypergraph $\mH$
and a sequence $\mG_n$ of $\mH$-free $k$-uniform hypergraphs with $n$ vertices
such that the size of the largest clique or independent set in $\mG_n$ is at
most $c_k r_k^{-1}(n)$.
\end{theorem}

In the next two sections, we will prove Theorem \ref{TripIntro}. In section
\ref{StepUp}, we will study properties of stepping-up lemmas in order to prove
Theorem \ref{StepUpIntro}. We conclude with several remarks in section
\ref{Concluding}. Throughout the paper, we systematically omit floor and
ceiling signs whenever they are not crucial for the sake of
clarity of presentation. We also do not make any serious attempt
to optimize absolute constants in our statements and proofs. All logarithms,
unless otherwise stated, are taken to the base $e$. As is quite customary in
Ramsey theory, our approach to the problems in this paper is asymptotic in
nature. We thus assume that the underlying parameter (normally the order of the
graph/hypergraph) is sufficiently large whenever necessary.

\section{The embedding lemma} \label{EmbedSec}

The {\it edge density} $d(X,Y)$ between two disjoint vertex subsets $X,Y$ of a
graph $G$ is the fraction of pairs $(x,y) \in X \times Y$ that are edges of
$G$. That is,
$d(X,Y)=\frac{e(X,Y)}{|X||Y|}$, where $e(X,Y)$ is the number of edges with one
endpoint in $X$ and the other in $Y$.

\definition A graph $G$ on $n$ vertices is said to be bi-$(\e, \rho)$-dense if,
for all $X$ and $Y$ with $X \cap Y = \emptyset$ and $|X|, |Y| \geq \e n$, the
density of edges $d(X,Y)$ between $X$ and $Y$ satisfies $d(X, Y) \geq \rho$.

\vspace{3mm}
This definition, introduced by Graham, R\"odl and Ruci\'nski \cite{GRR00}, has
proved very useful in Ramsey theory (see, for example, \cite{C11, CFS11, FS209,
GRR01, S11}). The reason for this is that there is an embedding lemma which
says that for any given $H$ and $\rho$ there exists an $\e$ such that if $G$
is a sufficiently large bi-$(\e, \rho)$-dense graph then $G$ contains a copy of
$H$. More generally, if both $G$ and its complement $\overline{G}$ are bi-$(\e,
\rho)$-dense, for $\e$
sufficiently small, then $G$ will contain an induced copy of $H$. Moreover, the
dependency of $\e$ on $\rho$ is polynomial, that is, $\e = \rho^{c(H)}$ (up to
a constant factor, which we
ignore for convenience). To see why this is helpful, by taking $\rho =
n^{-1/2c(H)}$, we have that
if $G$ and $\overline{G}$ are both bi-$(n^{-1/2}, \rho)$-dense, then there is
an induced copy of $H$ in $G$. Suppose, on the other hand, that either $G$ or
$\overline{G}$ is not bi-$(n^{-1/2}, \rho)$-dense. Without loss of generality,
we may assume that $G$ is the bad graph. This implies that there are two
disjoint subsets $X$ and $Y$ with $|X|, |Y| \geq n^{1/2}$ and density less than
$\rho$ in $G$. In the complement $\overline{G}$, this says that the density
between $X$ and $Y$ is at least $1 - \rho$. By the choice of $\rho$, this in
turn implies that there is an empty bipartite graph with parts each of
polynomial size.

We would like to imitate this argument for $3$-uniform hypergraphs. There are
several steps. Firstly, we need to find the correct notion of tri-density. We
then need to show that this yields an appropriate embedding lemma. Finally, we
need to show that if a hypergraph is not tri-dense, then it contains a large
empty tripartite subgraph. We will deal with the first two of these in this
section.

Roughly speaking, a graph is bi-dense if, between any two large vertex sets,
there are many edges. It is tempting to define tri-density in a similar
fashion, by saying that between any three large vertex sets there should be
many edges. Unfortunately, there are examples of $3$-uniform hypergraphs which
are tri-dense in this sense but do not contain any copies of $K_4^{(3)}$ or
even $K_4^{(3)}$ with one edge removed (see, for example, \cite{KNRS10}).
Instead, it is necessary to define tri-density with respect to edge sets. This
is similar to the sort of quasirandomness which arises in the study of
hypergraph regularity \cite{G06, G07, NRS06, RS04}. Formally, we have the
following definition.

\definition A $3$-uniform hypergraph $\mG$ on $n$ vertices is said to be
tri-$(\e, \rho)$-dense if for any three disjoint vertex subsets $V_1, V_2$ and
$V_3$ and any triple of graphs $G_{12}, G_{23}$ and $G_{31}$, with $G_{ij}$
between $V_i$ and $V_j$, for which there are at least $\e n^3$ triangles with
one edge in each of the $G_{ij}$, a $\rho$-proportion of these triangles form
edges in $\mG$.

\vspace{3mm}
We now prove an embedding lemma to complement this definition. We will need a
notion of bi-density defined for bipartite graphs.

\definition A bipartite graph $G$ between sets $U$ and $V$ is said to be
bi-$(\e, \rho)$-dense if, for all $X$ and $Y$ with $X \subseteq U$, $Y
\subseteq V$ and $|X| \geq \e |U|$, $|Y| \geq \e |V|$, the density of edges
$d(X,Y)$ between $X$ and $Y$ satisfies $d(X, Y) \geq \rho$.

\begin{lemma} \label{Embedding}
Let $\mH$ be a complete $3$-uniform hypergraph on $t \geq 3$ vertices $v_1,
\dots,
v_t$, where each edge has been assigned a colour from the set $\{1, 2, \dots,
\ell\}$, and let $\rho > 0$ be a positive constant. Let $\e = (2t)^{-10}
\rho^{3 t^2}$. Then, if $\mG$ is an $\ell$-coloured $3$-uniform hypergraph on
$n \geq (2t)^{10} \rho^{-3 t^2}$ vertices which is tri-$(\e, \rho)$-dense in
each of the $\ell$ colours, $\mG$ contains a copy of $\mH$.
\end{lemma}

\begin{proof}
Let the vertices of $\mH$ be $v_1, v_2, \dots, v_t$ and let $\chi:
\binom{V(\mH)}{3} \rightarrow \{1, 2, \dots, \ell\}$ be the colouring of the
edges of $\mH$. Split the vertex set $V(\mG)$ into $t$ vertex sets $U_1, \dots,
U_t$ each of size $n/t$. We will embed the graph $\mH$ one vertex at a time,
embedding $v_i$ into $f(v_i) \in U_i$. We will prove by induction on $i$ that
when vertices $v_1, v_2, \dots, v_i$ have been embedded, there are vertex sets
$U_j^i$, for $j > i$, and graphs $G_{jk}^i$, for $i < j < k \leq t$,  such that
the following conditions hold.
\begin{enumerate}
\item
$|U_j^i| \geq c_i n$ with $c_i = \frac{1}{t} \rho^{i^2/2}$;

\item
$G_{jk}^i$ is a graph between $U_j^i$ and $U_k^i$ which is bi-$(\e_i,
\rho^i)$-dense with $\e_i = \frac{1}{2t^2} \rho^{t^2 - i^2/2}$;

\item
for every $h \leq i$, every edge in $G_{jk}^i$ forms an edge of $\mG$ of colour
$\chi(h,
j, k)$ with $f(v_h)$ and, for all $h_1 < h_2 \leq i$, every vertex in $U_j^i$
forms an edge of $\mG$ of colour $\chi(h_1, h_2, j)$ with $f(v_{h_1})$ and
$f(v_{h_2})$.

\end{enumerate}
For $i = 0$, we let $U_j^0 = U_j$ and $G_{jk}^0$ be the complete graph between
$U_j$ and $U_k$. The three conditions then hold trivially. Suppose, therefore,
that we have embedded vertices $v_1, v_2, \dots, v_i$ while maintaining
conditions $1, 2$ and $3$ and that we now wish to embed $v_{i+1}$.

Let $W_{i+1}$ be the set of vertices in $U_{i+1}^i$ such that the neighborhood
$U_j^i(w)$ of $w$ in the graph $G_{i+1, j}^i$, for each $i+2 \leq j \leq t$,
has size at least $\rho^i |U_j^i|$. By the bi-$(\e_i,\rho^i)$-density condition
on $G_{i+1, j}^i$,
there are at most $\e_i |U_{i+1}^i| \leq \e_i n$ vertices in  $U_{i+1}^i$ with
fewer than $\rho^i
| U_j^i|$ neighbours in $U_j^i$. Adding over all $i+2 \leq j \leq t$, we have,
since $c_i \geq 2t \e_i$, that
\[|W_{i+1}| \geq |U_{i+1}^i| - t \e_i n \geq \frac{c_i}{2} n.\]

Given a vertex $w \in W_{i+1}$ and $i + 2 \leq j < k \leq t$, let $H_{jk}(w)$
be the subgraph of $G_{jk}^i$ between $U_j^i(w)$ and $U_k^i(w)$ consisting of
those edges $xy$ with $\chi(w, x, y) = \chi(i+1, j, k)$. Let $W_{jk}^{i+1}$ be
the set of vertices $w \in W_{i+1}$ such that the graph $H_{jk}(w)$ is not
bi-$(\epsilon_{i+1}, \rho^{i+1})$-dense between $U_j^i(w)$ and $U_k^i(w)$. Note
that for any $w \in W_{jk}^{i+1}$ there exist subsets $Y_j(w) \subseteq
U_j^i(w)$ and $Y_k(w) \subseteq U_k^i(w)$ such that $|Y_j(w)| \geq \e_{i+1}
|U_j^i(w)|$,  $|Y_k(w)| \geq \e_{i+1}
| U_k^i(w)|$ and the density of $H_{jk}(w)$ between them is less than
$\rho^{i+1}$.

Let $J_{i+1,j}$ be the graph between $W_{jk}^{i+1}$ and $U_j^i$ connecting $w$
to $Y_j(w)$ and let $J_{i+1, k}$ be defined similarly. Note that
\begin{equation}\label{123abc}|Y_j(w)| \geq \e_{i+1} |U_j^i(w)| \geq \e_{i+1}
\rho^i |U_j^i| \geq \e_i
| U_j^i|\end{equation}
and, similarly, $|Y_k(w)| \geq \e_i |U_k^i|$. Since the
graph $G_{jk}^i$ is bi-$(\e_i, \rho^i)$-dense between $U_j^i$ and $U_k^i$, the
tripartite graph between
$W_{jk}^{i+1}$, $U_j^i$ and $U_k^i$ where the subgraphs are $J_{i+1,j}$,
$J_{i+1,k}$ and $G_{jk}^i$ has at least
$\rho^i \sum_{w \in W_{jk}^{i+1}} |Y_j(w)| |Y_k(w)|$ triangles. If
$|W_{jk}^{i+1}| \geq \frac{|W_{i+1}|}{2 t^2} \geq \frac{c_i}{4 t^2} n$, by
(\ref{123abc}), this is
at least
\[\rho^i \e_i^2 |W_{jk}^{i+1}||U_j^i||U_k^i| \geq \rho^i \e_i^2 c_i^2
| W_{jk}^{i+1}|n^2 \geq \frac{\rho^i \e_i^2 c_i^3}{4 t^2} n^3 \geq \e n^3\]
triangles. Therefore, by the definition of tri-density, at least a
$\rho$-proportion of these triangles, that is, at least $\rho^{i+1} \sum_{w \in
W_{jk}^{i+1}} |Y_j(w)| |Y_k(w)|$ triangles will be in the colour $\chi(i+1, j,
k)$.

On the other hand, the number of $3$-uniform edges in colour $\chi(i+1,j,k)$
which contain $w$ and have one edge in each of $J_{i+1,j}$, $J_{i+1,k}$ and
$G_{jk}^i$ is the number of edges in $H_{jk}(w)$ between $Y_j(w)$ and $Y_k(w)$.
By definition of $Y_j(w)$ and $Y_k(w)$, this is less than $\rho^{i+1} |Y_j(w)|
| Y_k(w)|$. Therefore, the total number of $3$-uniform edges with colour
$\chi(i+1,j,k)$ in the tripartite graph is less than $\rho^{i+1} \sum_{w \in
W_{jk}^{i+1}} |Y_j(w)| |Y_k(w)|$. This is a contradiction. We must therefore
have that $|W_{jk}^{i+1}| < \frac{|W_{i+1}|}{2t^2}$.

Note therefore that the number of vertices $w$ in $W_{i+1}$ which are not in
$W_{jk}^{i+1}$ for any $i+2 \leq j < k \leq t$ is at least $|W_{i+1}|-{t-i-1
\choose 2}\frac{|W_{i+1}|}{2t^2} \geq \frac{|W_{i+1}|}{2} \geq \frac{c_i}{4}
n$. Let $f(v_{i+1}) = w_{i+1}$ be any vertex from this set. For $i+2 \leq j
\leq t$, let $U_j^{i+1} = U_j^i(w_{i+1})$ and, for $i+2 \leq j < k \leq t^2$,
let $G_{jk}^{i+1} = H_{jk}(w_{i+1})$. Note that $|U_j^{i+1}| \geq \rho^i
| U_j^i| \geq \rho^i c_i n \geq c_{i+1} n$ and $G_{jk}^{i+1}$ is bi-$(\e_{i+1},
\rho^{i+1})$-dense between $U_j^{i+1}$ and $U_k^{i+1}$.

Finally, by definition, for every edge $xy$ in $G_{jk}^{i+1} =
H_{jk}(w_{i+1})$,
$\chi(w_{i+1}, x, y) = \chi(i+1,j,k)$ and, for every $h \leq i$ and every $x
\in U_j^{i+1}$, $\chi(f(v_h), w_{i+1}, x) = \chi(h, i+1,j)$. Therefore, all 3
conditions are satisfied and the result follows by induction.
\end{proof}

The particular case where $\ell = 2$ implies that for any given $\mH$ and
$\rho$ there is a constant $\e$, polynomial in $\rho$, such that if $\mG$ and
its complement
$\overline{\mG}$ are tri-$(\e, \rho)$-dense then $\mG$ contains an induced
copy of $\mH$. This is all we will require to prove Theorem \ref{TripIntro}.

\section{A tripartite Erd\H{o}s-Hajnal theorem} \label{EHTrip}

The problem of Zarankiewicz \cite{Z51} asks for the maximum number $z(m,n;s,t)$
of edges in a bipartite graph $G$ which has $m$ vertices in the first class,
$n$ vertices in the second and does not contain a complete bipartite subgraph
$K_{s,t}$ with $s$ vertices in the first class and $t$ in the second. In their
celebrated paper, K\H{o}v\'ari, S\'os and Tur\'an \cite{KST54} used double
counting together with the pigeonhole principle to give a general upper bound
on $z(m,n;s,t)$. Using this technique, we obtain the following simple lemmas
which we will need to analyse $3$-uniform hypergraphs which are not tri-dense.
The degree $d(v)$ of a vertex $v$ is the number of vertices adjacent to $v$.

\begin{lemma} \label{Zarank}
If $G$ is a bipartite graph between sets $A$ and $B$ with at least $\e |A||B|$
edges and such that $s^{3/2} \leq \frac{\e}{2} |A|$, then it contains a copy of
$K_{s,t}$ with $t = e^{-s^{1/2}} \e^s |B|$ for which the set of $s$ vertices is
in $A$ and the set of $t$ vertices is in $B$.
\end{lemma}

\begin{proof}
The number of pairs $(U, v)$ with $U$ being a subset of $A$ of size $s$ and $v$
being a vertex in $B$ adjacent to every vertex of $U$ is at least
\[\sum_{v \in B} \binom{d(v)}{s} \geq |B| \binom{\e |A|}{s} \geq |B| \frac{(\e
| A|)^s}{s!} e^{-s^{1/2}},\]
where the inequalities follow from the convexity of $f(x) = \binom{x}{s}$ and
the fact that, for $s^{3/2} \leq \frac{\e}{2} |A|$ and $x = \e|A|$,
\[x (x - 1) \dots (x - s + 1) \geq x^s \prod_{i=1}^{s-1} \left(1 -
\frac{i}{x}\right) \geq x^s e^{- 2 \sum_{i=1}^{s-1} \frac{i}{x}} \geq x^s
e^{-s^2/x} \geq x^s e^{-s^{1/2}}.\]
Here we used that $1 - z \geq e^{-2z}$ for $0 \leq z \leq \frac{1}{2}$. If $G$
does not contain $K_{s,t}$ then we know that every subset of $A$ of size $s$
has at most $t - 1$ common neighbours. Therefore,
\[e^{-s^{1/2}}|B| \frac{(\e |A|)^s}{s!} \leq (t - 1) \binom{|A|}{s} < t
\frac{|A|^s}{s!} = e^{-s^{1/2}} |B| \frac{(\e |A|)^s}{s!},\]
a contradiction.
\end{proof}

\begin{lemma} \label{Zarank2}
Let $G$ be a bipartite graph with parts $A$ and $B$ and at least $\e |A||B|$
edges. Then $G$ contains a complete bipartite subgraph $K_{s,t}$ with $s = \e
| A|$ vertices from $A$ and $t = 2^{-|A|}|B|$ vertices from $B$.
\end{lemma}

\begin{proof}
By the convexity of the function $f(x) = \binom{x}{s}$ and the fact that the
average degree of $A$ is at least $s$, we conclude that the number of pairs
$(U,v)$ with $U$ a subset of $A$ of size $s$ and $v$ a vertex in $B$ connected
to every element of $U$ is at least
\[\sum_{v \in B} \binom{d(v)}{s} \geq |B| \binom{\frac{1}{|B|} \sum_{v \in B}
d(v)}{s} \geq |B|.\]
Since $A$ has at most $2^{|A|}$ subsets, the pigeonhole principle implies that
for some $U \subset A$ of size $s$ there are at least $t = 2^{-|A|}|B|$
elements $b$ of $B$ which are connected to every element of $U$. This yields
the required copy of $K_{s,t}$.
\end{proof}

The following lemma, which we believe to be of independent interest, says that
if a graph $G$ contains many triangles, a $3$-uniform hypergraph $\mG$ whose
edges form a dense subset of the set of triangles in $G$ contains a larger copy
of $K_{s,s,s}^{(3)}$ than one could normally expect in a $3$-uniform hypergraph
of the same density.

\begin{lemma} \label{NotTriDense}
Suppose that $V_1, V_2$ and $V_3$ are disjoint vertex sets of size at most $n$
and, for
$1 \leq i < j \leq 3$, there is a bipartite graph $G_{ij}$ between $V_i$ and
$V_j$. Suppose that there are at least $\d n^3$ triangles in this tripartite
graph. Suppose further that $\mG$ is a $3$-uniform hypergraph which contains a
$(1 - \eta)$-proportion of the triangles in the tripartite graph, where $0 <
\eta \leq \frac{1}{8}$. Then $\mG$ contains a copy of $K_{s, s, s}^{(3)}$,
provided that
\[e^{2^{10} \d^{-2} s^{3/2}} (1 - 4 \eta)^{-4s^2}
\left(\frac{16}{\d}\right)^{4s} \leq n.\]
\end{lemma}

\begin{proof}
We suppose without loss of generality that $|V_1|=|V_2|=|V_3|=n$. Indeed, we
can otherwise add isolated vertices to make each part have size $n$, and the
hypothesis and conclusion of the lemma are not affected.  For every edge $e$ in
$G_{12}$, let $\Delta_e$ be the number of triangles
containing $e$. We will say a particular edge $e$ is good if the number of
edges of $\mG$ containing $e$ is a $(1 - 2 \eta)$-proportion of the total
number of triangles containing $e$. Otherwise, we say an edge is bad. Note that
\[\sum_{e \ good} \Delta_e + \sum_{e \ bad} (1 - 2 \eta) \Delta_e \geq (1 -
\eta) \Delta,\]
where $\Delta$ is the total number of triangles in the graph. Therefore, since
$\sum_e \Delta_e = \Delta$,
\[\sum_{e \ bad} \Delta_e \leq \frac{1}{2} \Delta,\]
that is, at least half the triangles in the tripartite graph contain a good
edge. Since there are at least $\d n^3$ triangles in the tripartite graph, this
tells us that there are at least $\frac{\d}{2} n^3$ triangles containing good
edges. Calling the set of good edges $G'_{12}$, we see that $G'_{12}$ has at
least $\frac{\d}{2} n^2$ edges.

Let $H_{12}$ be the subgraph of $G'_{12}$ consisting of edges which are
contained in at least $\frac{\d}{4} n$ triangles from the tripartite graph.
Then $H_{12}$ has at least $\frac{\d}{4} n^2$ edges. Otherwise, there would be
fewer than $\frac{\d}{4}n \cdot n^2 + \frac{\d}{4} n^2 \cdot n =\frac{\d}{2}
n^3$ triangles containing good edges, which would be a contradiction. Let $s_1
= \frac{\d}{16} \log n$ and $t_2 = n^{\frac{1}{4}}$. Then, since $s_1^{3/2}
\leq \frac{\d}{8} n$ and
\[e^{-s_1^{1/2}} \left(\frac{\d}{4}\right)^{s_1} n \geq e^{-\frac{1}{4} \log n}
n^{-\frac{\d}{16} \log(4/\d)} n \geq n^{-\frac{1}{4}} n^{-\frac{1}{4}} n \geq
n^{\frac{1}{4}} = t_2,\]
we may apply Lemma \ref{Zarank} to the graph $H_{12}$ to find a set $S_1$ of
size $s_1$ in $V_1$ the elements of which have at least $t_2$ common neighbors
in $V_2$. Call this set of neighbors $T_2$.

We will now restrict our attention to the tripartite graph between the vertex
sets $S_1$, $T_2$ and $V_3$. By the choice of $H_{12}$, there are at least
$\frac{\d}{4} |S_1| |T_2| |V_3|$ triangles in this graph. Moreover, since each
of the edges between $S_1$ and $T_2$ is good, at least a $(1 - 2
\eta)$-proportion of these triangles are triples in the hypergraph $\mG$.

We may now repeat the same process within this new tripartite graph. We say
that an edge of the graph between $S_1$ and $V_3$ is good if, within the new
graph, the number of triples in $\mG$ containing $e$ is a $(1 - 4
\eta)$-proportion of the total number of triangles containing $e$.
Then, by the same argument as before, at least half the triangles in the new
tripartite graph contain good edges. If we let $G'_{13}$ be the set of good
edges, we see that this set has at least $\frac{\d}{8} |S_1| |V_3|$ edges.
Moreover, letting $H_{13}$ be the subgraph of $G'_{13}$ consisting of edges
which are contained in at least $\frac{\d}{16} |T_2|$ triangles from the
tripartite graph, we see that $H_{13}$ contains at least $\frac{\d}{16} |S_1|
| V_3|$ edges.
Let $s_2 = \left(\frac{\d}{16}\right)^2 \log n$ and $t_3 = n^{\frac{1}{4}}$.
Since $s_2 = \frac{\d}{16} s_1$ and $2^{-s_2} n \geq t_3$, we may apply
Lemma \ref{Zarank2} to the graph $H_{13}$ to find a set $S_2$ of size $s_2$ in
$S_1$ the elements of which have at least $t_3$ common neighbors in $V_3$. Call
this set of neighbors $T_3$.

We now have a tripartite graph between sets $S_2$, $T_2$ and $T_3$ such that
the graphs between $S_2$ and $T_2$ and $S_2$ and $T_3$ are complete. Moreover,
the graph contains at least $\frac{\d}{16} |S_2| |T_2| |T_3|$ triangles, a
$(1-4\eta)$-proportion of which are triples in the hypergraph $\mG$.

Consider the induced subgraph of $G_{23}$ between $T_2$ and $T_3$. Let the set
of edges in this graph be $E_{23}$. Note that this graph has density at least
$\frac{\d}{16}$. We consider the bipartite graph $K$ between $S_2$ and
$E_{23}$, where $v$ and $e$ are connected if together they span a triple in
$\mG$. Since a $(1 - 4 \eta)$-proportion of the triangles between $S_2$, $T_2$
and $T_3$ are triples, the graph $K$ must have at least $(1-4\eta)
| S_2||E_{23}|$ edges.
Note that, since
\[e^{2^{10} \d^{-2} s^{3/2}} \leq n,\]
we have $s^{3/2} \leq \frac{1}{4} \left(\frac{\d}{16}\right)^2 \log n \leq
\frac{1}{2} (1 - 4 \eta) s_2$. Therefore, an application of Lemma \ref{Zarank}
implies that there is a set $S$ of $s$ vertices in $S_2$ which are connected to
$e^{-s^{1/2}} (1 - 4 \eta)^s |E_{23}|$ edges in $E_{23}$. This yields a
subgraph $H_{23}$ of $G_{23}$ between $T_2$ and $T_3$ of density at least $\e_0
=
e^{-s^{1/2}} (1 - 4 \eta)^s \frac{\d}{16}$ such that every vertex in $S$ is
connected to every edge in $H_{23}$.

Note that
\[s^{3/2} e^{s^{1/2}} \e_0^{-s} \leq e^{3 s^{3/2}} (1-4\eta)^{-s^2}
\left(\frac{16}{\d}\right)^s \leq n^{\frac{1}{4}}.\]
Therefore, $s^{3/2} \leq \frac{\e_0}{2} t_2$ and $e^{-s^{1/2}} \e_0^s t_3 \geq
s$. Applying Lemma \ref{Zarank} to the graph $H_{23}$ yields a complete
subgraph between two subsets $S'$ and $S''$, each of size $s$, of $V_2$ and
$V_3$. The $3$-uniform hypergraph between $S$, $S'$ and $S''$ is the required
$K_{s,s,s}^{(3)}$.
\end{proof}

We are now ready to put Lemmas \ref{Embedding} and \ref{NotTriDense} together
to prove Theorem \ref{TripIntro} in the following precise form.

\begin{theorem} \label{TripBody}
Let $\mH$ be a fixed $3$-uniform hypergraph with $t$ vertices. Then any
$\mH$-free $3$-uniform hypergraph on $n$ vertices with $n$ sufficiently large
contains a complete or empty tripartite subgraph each part of which has order
at least $ (\log n)^{\frac{1}{2} + \d(\mH)}$, where $\d(\mH) =
1/(55t^2)$.
\end{theorem}

\begin{proof}
Let $\mG$ be a $3$-uniform hypergraph on $n$ vertices. Let $\rho =  (\log
n)^{1/(27t^{2})}$ and $\e = (2t)^{-10} \rho^{3 t^2} = \Omega \left( (\log
n)^{-\frac{1}{9}}\right)$. If both the graph $\mG$ and its complement
$\overline{\mG}$ are tri-$(\e, \rho)$-dense, then we may apply Lemma
\ref{Embedding} to conclude that $\mG$ contains an induced copy of $\mH$.

We may therefore assume that either $\mG$ or its complement is not tri-$(\e,
\rho)$-dense. Without loss of generality, we will assume that $\overline{\mG}$
is not tri-$(\e, \rho)$-dense, that is, that there exist three disjoint vertex
sets $V_1, V_2$ and $V_3$ and bipartite graphs $G_{12}, G_{23}$ and $G_{31}$,
with $G_{ij}$ between $V_i$ and $V_j$, such that the number of triangles with
one edge in each of the $G_{ij}$ is at least $\e n^3$ but the number of triples
of $\overline{\mG}$ contained within these triangles is less than a
$\rho$-fraction of these triangles. Taking the complement, we see that at least
a $(1-\rho)$-fraction of the triangles are edges of $\mG$.

Let $s = (\log n)^{\frac{1}{2} + \d(\mH)}$, where $\d(\mH) = 1/(55t^2)$.
Note that, since $\rho \leq \frac{1}{8}$ and $1 - z \geq e^{-2z}$ for $0 \leq z
\leq \frac{1}{2}$,
\begin{equation}\label{another} e^{2^{10} \e^{-2} s^{3/2}} (1 - 4 \rho)^{-4s^2}
\left(\frac{16}{\e}\right)^{4s} \leq e^{2^{10} \e^{-2} s^{3/2}} e^{32 \rho s^2}
e^{64s/\e}.\end{equation}
This expression has three terms, which we consider in turn. The exponent of the
first is
\[2^{10} \e^{-2} s^{3/2} =O\left((\log n)^{2 \cdot \frac{1}{9}+\frac{3}{2}
\cdot \big(\frac{1}{2}+1/(55t^2)\big)}\right)=o(\log n).\]
The exponent of the second term is
\[32 \rho s^2 =O\left((\log
n)^{-\frac{1}{27t^{2}}+2\big(\frac{1}{2}+1/(55t^2)\big)}\right)=o(\log n).\]
The exponent of the third term is
\[\frac{64 s}{\e} =O\left(\log
n)^{\frac{1}{2}+1/(55t^2)+\frac{1}{9}}\right)=o(\log n).\]
Hence, as $n$ is sufficiently large, overall the expression in (\ref{another})
is $n^{o(1)} \leq n$.
Since the number of
triangles with one vertex in each of the parts $V_1,V_2,V_3$ is at least $\e
n^3$, and each of the parts has order at most $n$, an application of Lemma
\ref{NotTriDense}
with $\eta = \rho$ and $\d = \e$ tells us that the graph contains a copy of
$K_{s,s,s}^{(3)}$ with $s =  (\log n)^{\frac{1}{2} + \d(\mH)}$,
where $\d(\mH) = 1/(55t^2)$.
\end{proof}

Erd\H{o}s and Hajnal \cite{EH89} also considered the case where the edges of a
complete graph $K_n$ have been $\ell$-coloured and some fixed coloured subgraph
$H$ is banned. It is now too much to hope that there might be a large clique in
one particular colour. Instead, the natural object to look for is a large
clique which avoids one particular colour. In this case, Erd\H{o}s and Hajnal
showed that there is a clique of size $e^{c \sqrt{\log n}}$, where $c$ depends
only on $H$ and $\ell$. Moreover, their methods also allow one to find a
bipartite graph with polynomial sized parts which avoids a particular colour.
The following $3$-uniform analogue has essentially the same proof as Theorem
\ref{TripBody}. We omit the details.

\begin{theorem}
Let $\mH$ be a complete $3$-uniform hypergraph such that each edge has been
assigned a colour from the set $\{1, 2, \dots, \ell\}$. Then there exists a
constant $\d(\mH) > 0$ such that, for $n$ sufficiently large, any
$\ell$-coloured $3$-uniform hypergraph on $n$ vertices which does not contain a
coloured copy of $\mH$ must contain a complete tripartite subgraph which avoids
a particular colour class such that each part has order at least $(\log
n)^{\frac{1}{2} + \d(\mH)}$.
\end{theorem}

\section{Some properties of step-up colourings} \label{StepUp}

In order to show that step-up colourings do not contain certain subgraphs, we
must first know what these colourings look like. Assume, therefore, that the
edges of the $k$-uniform hypergraph $K_n^{(k)}$ have been red/blue-coloured.
Let
$$T=\{(\gamma_1,\ldots,\gamma_n):\gamma_i=0 \mbox{ or } 1\}.$$

If $\epsilon = (\gamma_1, \cdots, \gamma_n)$, $\epsilon' =
(\gamma'_1, \cdots, \gamma'_n)$ and $\epsilon \neq \epsilon'$,
define
\[\delta(\epsilon, \epsilon') = \max\{i : \gamma_i \neq
\gamma'_i\},\] that is, $\delta(\epsilon, \epsilon')$ is the
largest coordinate at which they differ. Given this, we can define
an ordering on $T$, saying that
\[\epsilon < \epsilon' \mbox{ if } \gamma_i = 0, \gamma'_i = 1,\]
\[\epsilon' < \epsilon \mbox{ if } \gamma_i = 1, \gamma'_i = 0,\]
where $i=\delta(\epsilon,\epsilon')$. Equivalently, associate to any $\epsilon$
the number $b(\epsilon) = \sum_{i=1}^n \gamma_i 2^{i-1}$. The ordering then
says simply
that $\epsilon < \epsilon'$ if and only if $b(\epsilon) < b(\epsilon')$.

We colour the complete $(k+1)$-uniform hypergraph on the set $T$ as follows. If
$\epsilon_1 < \ldots < \epsilon_{k+1}$,  for $1 \leq i \leq k$, let $\delta_i =
\delta(\epsilon_i, \epsilon_{i+1})$. A fundamental property of step-up
colourings is that $\delta_i$ is different from $\delta_{i+1}$.

If $\delta_1,\ldots,\delta_k$ form a monotone sequence (increasing or
decreasing),
then let the colour of $\{\epsilon_1,\ldots,\epsilon_{k+1}\}$ be given by the
colour of $\{\delta_1,\ldots,\delta_k\}$.

If $\delta_1,\ldots,\delta_k$ is not monotone, it must contain a local maximum
or local minimum, that is, an $i$ for which $\delta_{i-1} < \delta_i >
\delta_{i+1}$ or $\delta_{i-1} > \delta_i < \delta_{i+1}$ respectively. Let
$\delta_j$ be the first such local extremum, that is, the one with the smallest
subscript. We colour $\{\epsilon_1, \ldots,\epsilon_{k+1}\}$ blue or red,
respectively, depending on whether the sequence $\delta_1,\ldots,\delta_k$ has
a local maximum or minimum at $\delta_j$.

This colouring is exactly that used by Erd\H{o}s and Hajnal to prove their
stepping-up lemma. We will refer to such a colouring as a step-up colouring. In
order to prove Theorem \ref{StepUpIntro},  we will show that there are certain
$2$-coloured hypergraphs $\mH$
which do not occur within step-up colourings.

\begin{lemma}\label{stpup}
For $k \geq 3$, there is a fixed $2$-coloured complete $(k+1)$-uniform
hypergraph $\mH$
which never occurs within a step-up colouring $T$.
\end{lemma}

\begin{proof}
We will in fact show that for every $h \geq k+5$ there is a $2$-coloured complete
$(k+1)$-uniform hypergraph on $h$ vertices which never occurs within a step-up
colouring. We accomplish this by a simple counting argument which shows that the
number of colourings on $h$ vertices which occur within a step-up colouring is
considerably less than the total number of colourings on $h$ vertices. We first
bound the number of colourings on $h$ vertices which occur within a step-up
colouring.

A total preorder $\leq$ is a binary relation that is transitive (i.e., for all
$x$, $y$, and $z$, if $x \leq y$ and $y \leq z$, then $x \leq z$) and total
(i.e., for all $x$ and $y$, $x \leq y$ or $y \leq x$). For example, the
elements of a sequence of integers has the total preorder $\leq$. In
particular, we can associate to each  sequence $\delta_1,\ldots,\delta_{h-1}$
its total preorder $\leq$.  The ordered Bell number $H_n$ is the number of
total preorders on a sequence of $n$ elements, and also counts the number of
ordered partitions of the set $[n]:=\{1,\ldots,n\}$. For each ordered partition
$P:[n]=I_1 \cup \ldots \cup I_t$, consider the mapping $f_P:[n] \rightarrow
[n]$ defined by, for $i \in I_j$, $f_P(i)=j$. This gives an injective mapping
from the set of ordered partitions of $[n]$ to the set of mappings from $[n]$
to $[n]$. Hence, $H_n \leq n^n$.

For a sequence $\epsilon_1<\ldots<\epsilon_h$, a useful property is that, for
$1 \leq i<j \leq h$,  we have
\begin{equation}\label{deltadefine}\delta_{i,j}:=\delta(\epsilon_i,\epsilon_j)=\max_{i
\leq l \leq j-1} \delta_{l},\end{equation}
where $\delta_{l}=\delta(\epsilon_{l},\epsilon_{l+1})$. In particular, the
${h \choose 2}$ integers $\delta_{i,j}$ are completely determined by the $h-1$
integers $\delta_i$.

We claim that the colours of the $(k+1)$-tuples of
$\epsilon_1,\ldots,\epsilon_h$ in a step-up colouring are completely determined
by the colours of the $k$-tuples of $\delta_1,\ldots,\delta_{h-1}$ and the
total preorder $\leq$ on these numbers. In particular, the colouring of the
$(k+1)$-tuples does not depend on the $\delta$s, but only on the colouring of
the edges between them and their order. To verify this claim, consider an edge
$e=(\epsilon_{i_1},\ldots,\epsilon_{i_{k+1}})$ with
$\epsilon_{i_1}<\ldots<\epsilon_{i_{k+1}}$. For $1 \leq i \leq k$, consider the
associated $k$-tuple $d=(\delta_{i_t,i_{t+1}})_{t=1}^k$, where the terms of
this sequence are defined in (\ref{deltadefine}).  If the sequence $d$ is not
monotone, then the colour of $e$ is determined by whether or not the first
local extrema of $d$ is a maxima or minima. This is determined by the order of
the elements of $d$, and it follows from (\ref{deltadefine}) that this is
determined by the total preorder $\leq$ on the $h-1$ numbers $\delta_i$.
Otherwise, $d$ is monotone, and the colour of $e$ is the colour of $d$, which
only depends on the colouring of the $k$-tuples.

The number of $2$-colourings of the $k$-tuples of a set of size $h-1$ is
$2^{h-1 \choose k}$. Hence, the  number of $2$-coloured complete
$(k+1)$-uniform hypergraphs on $h$ vertices which occur within a step-up
colouring is at most $H_{h-1}2^{h-1 \choose k} \leq (h-1)^{h-1}2^{h-1 \choose
k}:=A$.

The number of $2$-edge-coloured complete $(k+1)$-uniform hypergraphs on $h$
labeled vertices is $2^{h \choose k+1}$, and hence the number of distinct (up
to isomorphism)  $2$-edge-coloured complete $(k+1)$-uniform hypergraphs on $h$
vertices is at least $\frac{1}{h!}2^{h \choose k+1}:=B$. As for $h \geq k+5$,
we have
$$A/B=h!(h-1)^{h-1}2^{-\left(\frac{h-k-1}{k+1}\right){h-1 \choose k}}<1,$$
there is a $2$-coloured complete $(k+1)$-uniform hypergraphs on $h$ vertices
which does not appear in a step-up colouring.
\end{proof}

One can deduce Theorem \ref{StepUpIntro} from Lemma \ref{stpup} as follows.
Start with a $2$-edge-colouring of the complete $(k-1)$-uniform hypergraph on
$n$ vertices whose largest monochromatic clique is of order
$O(r_{k-1}^{-1}(n))$. Apply the stepping-up construction described above. By
the Erd\H{o}s-Hajnal result mentioned in the introduction, this gives a
$2$-edge-coloring of the complete $k$-uniform hypergraph on $N=2^n$ vertices
whose largest monochromatic clique is still of order
$O(r_{k-1}^{-1}(n))=O(r_k^{-1}(N))$. On, the other hand by Lemma \ref{stpup},
this colouring does not contain a fixed $2$-coloured complete $k$-uniform
hypergraph $\mH$.

\section{Concluding remarks} \label{Concluding}

A simple extension of the example given after Theorem \ref{TripIntro} allows
one to show that there are $k$-uniform hypergraphs $\mH$ and $\mH$-free
hypergraphs $\mG$ on $n$ vertices within which the largest complete or empty
$k$-partite subgraph, with all parts of the same order, is of order at most $c
(\log n)^{1/(k-2)}$. Indeed, this can be established by considering a random 
$(k-1)$-uniform hypergraph $\mG_0$ on the vertex set $\{1,\ldots,n\}$, and forming a $k$-uniform
hypergraph $\mG$ on the same vertex set whose edges are all $k$-tuples $(i_1,\ldots,i_k)$ with 
$i_1<\ldots<i_k$ such that $(i_1,\ldots,i_{k-1})$ is an edge of $\mG_0$. On the other hand, a standard density
argument shows that any hypergraph on $n$ vertices must contain a complete or
empty $k$-partite subgraph with parts of order $c (\log n)^{1/(k-1)}$. We make
the following conjecture.

\begin{conjecture}
Let $\mH$ be a $k$-uniform hypergraph, where $k \geq 3$. Then any $\mH$-free
$k$-uniform hypergraph on $n$ vertices contains a complete or empty $k$-partite
subgraph each part of which has order at least $(\log n)^{1/(k-2)-o(1)}$.
\end{conjecture}

It is possible to show that the notion of tri-density and the embedding lemma,
Lemma \ref{Embedding}, both extend to the $k$-uniform case. Indeed, if one is
not particularly concerned with bounds, an application of hypergraph regularity
and counting immediately allows one to prove such an embedding lemma. The real
stumbling block in extending our result to the $k$-uniform case is Lemma
\ref{NotTriDense}, which does not seem to generalise easily. Understanding and
resolving this difficulty could also be useful in extending the results of
\cite{CFS1102}, on finding almost monochromatic subsets within colourings of
$K_n^{(3)}$, to the $k$-uniform case.

Another direction which one could take would be to consider random $\mH$-free
hypergraphs. In a recent paper, Loebl, Reed, Scott, Thomason and Thomass\'e
\cite{LRSTT10} showed that almost every $H$-free graph on $n$ vertices contains
a clique or independent set of size $n^{\d(H)}$. Perhaps this result could be
extended to hypergraphs in some interesting fashion.

Finally, it would be of great interest to show that Theorem \ref{StepUpIntro}
also holds for $3$-uniform hypergraphs. It seems likely that in order to do
this one must first resolve the central Ramsey problem for $3$-uniform
hypergraphs. That is, one would need to show that there are $2$-colourings of
the edges of $K_n^{(3)}$ which contain no cliques or independent sets of size
$c \log \log n$. Such a colouring, if it does exist, is likely to avoid some
class of subgraphs.

\vspace{0.1cm} \noindent {\bf Acknowledgments.}\, Credit and thanks are due to
Vojta R\"odl and Mathias Schacht, who brought the question regarding tripartite subgraphs of
$\mH$-free hypergraphs to our attention.

\end{document}